\pdfoutput=1
\RequirePackage{ifpdf}
\ifpdf 
\documentclass[pdftex]{sigma}
\else
\documentclass{sigma}
\fi

\newcommand{\g}{\mathrm{g}}
\newcommand{\Y}{\mathrm{Y}}

\newcommand{\SO}{\text{SO}}
\newcommand{\R}{\mathbb{R}}
\newcommand{\E}{\mathbb{E}}

\usepackage{enumerate}

\numberwithin{equation}{section}

\newtheorem{Theorem}{Theorem}[section]
\newtheorem*{Theorem*}{Theorem}
\newtheorem{Lemma}[Theorem]{Lemma}
\newtheorem{Corollary}[Theorem]{Corollary}
\newtheorem{theoremA}{Theorem}

\theoremstyle{definition}
\newtheorem{Definition}[Theorem]{Definition}

\newtheorem{rem}[Theorem]{Remark}
\newtheorem{Example}[Theorem]{Example}

\begin{document}

\allowdisplaybreaks

\newcommand{\arXivNumber}{2602.00439}

\renewcommand{\thefootnote}{}

\renewcommand{\PaperNumber}{043}

\FirstPageHeading

\ShortArticleName{Finiteness of Totally Magnetic Hypersurfaces}

\ArticleName{Finiteness of Totally Magnetic Hypersurfaces\footnote{This paper is a~contribution to the Special Issue on Geometry and Dynamics in memory of Will Merry. The~full collection is available at \href{https://sigma-journal.com/Merry.html}{https://sigma-journal.com/Merry.html}}}

\Author{James MARSHALL REBER~$^{\rm a}$ and Ivo TEREK~$^{\rm b}$}

\AuthorNameForHeading{J.~Marshall Reber and I.~Terek}

\Address{$^{\rm a)}$~Department of Mathematics, University of Chicago, Chicago IL 60637, USA}
\EmailD{\mail{jmarshallreber@uchicago.edu}}

\Address{$^{\rm b)}$~Department of Mathematics, University of California Riverside, Riverside CA 92521, USA}
\EmailD{\mail{ivo.terek@ucr.edu}}

\ArticleDates{Received January 31, 2026, in final form April 26, 2026; Published online May 02, 2026}

\Abstract{By introducing a dynamical version of the second fundamental form, we generalize a recent result of Filip--Fisher--Lowe to the setting of magnetic systems. Namely, we show that a real-analytic negatively $s$-curved magnetic system on a closed real-analytic manifold has only finitely many closed totally $s$-magnetic hypersurfaces, unless the magnetic $2$-form is trivial and the underlying metric is hyperbolic.}


\Classification{37D20; 37D40; 53C15; 53C20}

\renewcommand{\thefootnote}{\arabic{footnote}}
\setcounter{footnote}{0}

\section{Introduction}

A \emph{magnetic system} on a~closed manifold $M$ is a pair $(\g,\sigma)$, where $\g$ is a Riemannian metric and~$\sigma$~is a~closed $2$-form on $M$. Associated to such a pair is a dynamical system on the tangent bundle~$TM$, called the \emph{magnetic flow}, which is given by the second-order nonhomogeneous differential equation
\begin{equation} \label{eqn:magnetic-geodesic} \frac{{\rm D} \dot{\gamma}}{{\rm d}t}(t) = \Y_{\gamma(t)}(\dot{\gamma}(t)),\end{equation}
where ${\rm D}/{\rm d}t$ denotes the covariant derivative operator along $\gamma$ induced by $\g$, and $\Y\colon TM\to TM$ is the \emph{Lorentz force operator} of $(\g,\sigma)$, defined via the condition
\begin{equation}\label{Lorentz-force} \g_x(\Y_x(v), w) = \sigma_x(v,w) \qquad\text{for all } x\in M\mbox{ and } v,w\in T_xM.\end{equation}
It is a straightforward exercise to show that the magnetic flow $\varphi^t\colon TM\to TM$, defined by \mbox{$\varphi^t(\gamma(0), \dot{\gamma}(0)) = (\gamma(t), \dot{\gamma}(t))$}, preserves each $s$-sphere bundle \begin{equation}\label{defn-Sigma_s}
 \Sigma_s = \bigl\{(x,v) \in TM \mid \g_x(v,v) = s^2\bigr\},\qquad s>0.
\end{equation}
Unlike the geodesic flow, the dynamics of the magnetic flow can change drastically as one varies~$s$. For example, if $(M,\g)$ is a hyperbolic surface and $\sigma$ is its area form, then for $s > 1$ the magnetic flow is continuously orbit-equivalent to the underlying geodesic flow, for $s = 1$ it is minimal, and for $s < 1$ every orbit is contractible (such example is well known, see \cite{paternain2006magnetic}). As a consequence, specifying the parameter $s$, we refer to the restriction of the magnetic flow to $\Sigma_s$ as the \emph{$s$-magnetic flow}.

Let $UM=\Sigma_1$ be the unit-tangent bundle of $(M,\g)$ and, for $(x,v)\in UM$, let ${P_v\colon T_xM \rightarrow \mathbb{R} v}$ and $P_{v^\perp}\colon T_xM \rightarrow v^\perp$ be the corresponding tangential and orthogonal projections. We also consider the vector bundle $E \rightarrow UM$ of orthogonal hyperplanes, whose fiber over an element $(x,v) \in UM$ is given by $E_{(x,v)} = v^\perp$. For $s > 0$, we consider the endomorphisms $A^{(g, \sigma)}$ and $R^{(g, \sigma, s)}$ of $E$ given by
\begin{equation}\label{R-and-A} \begin{split}
& A^{(\g,\sigma)}_{(x,v)}(w) = -\frac{3}{4} \Y_x(P_v(\Y_x(w))) - \frac{1}{4} P_{v^\perp}\bigl(\Y_x^2(w)\bigr), \\
& R^{(\g,\sigma,s)}_{(x,v)}(w) = s^2 R_x(w,v)v - s (\nabla_w \Y)_x(v) + \frac{s}{2} P_{v^\perp}((\nabla_v \Y)_x(w)),
\end{split} \end{equation}
where $\nabla$ is the Levi-Civita connection of $\g$ and $R$ its curvature tensor. The \emph{$s$-magnetic curvature operator} is then
\begin{equation*} M^{(g,\sigma,s)}_{(x,v)}(w) = R^{(\g,\sigma,s)}_{(x,v)}(w) + A^{(\g,\sigma)}_{(x,v)}(w)\end{equation*}and, if $\text{St}_2(M,\g)$ denotes the Stiefel bundle of ordered orthonormal $2$-frames tangent to $M$, the \emph{$s$-magnetic sectional curvature} \smash{$\text{Sec}^{(g,\sigma,s)} \colon \text{St}_2(M,\g) \rightarrow \mathbb{R}$} is \smash{$\text{Sec}^{(\g,\sigma,s)}_x(v,w) =\g_x\bigl(M^{(g,\sigma,s)}_{(x,v)}(w) , w\bigr)$}. See also \cite[Section 1.2]{assenza-imrn}.

By design, the $s$-magnetic curvature captures relevant dynamical information about the \mbox{$s$-magnetic} flow. For example, \begin{equation}\label{neg_sec_Anosov}
 \parbox{.7\textwidth}{if $\text{Sec}^{(g,\sigma,s)}<0$, then the $s$-magnetic flow is Anosov,}
\end{equation}cf.\
 \cite[Appendix A]{assenza2025magnetic}. Moreover, the magnetic sectional curvature also influences the underlying geometry. An example of this phenomenon is the following: if $\sigma \neq 0$ and $\dim(M) \geq 3$, and for some $s>0$ the $s$-magnetic sectional curvature is everywhere constant, then $\sigma$ is parallel and $J=\|\Y\|^{-1} \Y$ makes $(M,\g)$ a K\"ahler manifold with constant holomorphic sectional curvature equal to $-\|\Y\|^2/s^2$ \cite[Theorem D]{assenza2025magnetic}.

An immersed submanifold $N \subseteq M$ is \emph{totally $s$-magnetic} if every solution to \eqref{eqn:magnetic-geodesic} with speed~$s$ which starts tangent to $N$ remains in $N$ for small time, cf.\ \cite[Definition 6.5]{ABM_2025}. In light of the recent work by Filip, Fisher, and Lowe \cite[Theorem 1]{filip2024finiteness}, we aim to show that real-analytic magnetic systems of the form $(\g, 0)$, where $\g$ is a hyperbolic metric on an arithmetic manifold, are special in the sense that they are the only negatively magnetically curved examples which have infinitely many totally $s$-magnetic immersed hypersurfaces for any (and hence all) $s > 0$. As in \cite{filip2024finiteness}, a metric $\g$ is hyperbolic if it has constant strictly negative sectional curvature, a~manifold is hyperbolic if it admits a hyperbolic metric, and a hyperbolic manifold is arithmetic if it can be written as $\mathbb{H}^n/\Gamma$, where $\Gamma$ is an arithmetic lattice.

\begin{theoremA} \label{thm:magnetic-finiteness}
 Let $(\g, \sigma)$ be a real-analytic magnetic system on a closed, connected, real-analytic manifold $M$ of dimension at least $3$. If there exists $s > 0$ such that the $s$-magnetic sectional curvature of $(\g, \sigma)$ is everywhere negative and $M$ contains infinitely many closed totally \mbox{$s$-magnetic}
 hypersurfaces, then $\sigma = 0$ and $(M,\g)$ is isometric to a hyperbolic manifold. Moreover, $M$ is arithmetic.
\end{theoremA}

Our method of proof is dynamically motivated, and is similar in spirit to the work of Foulon in \cite{foulon}. In view of \eqref{neg_sec_Anosov}, by suitably rescaling both $\g$ and $\sigma$ if needed, we may assume that the $1$-magnetic flow of $(\g, \sigma)$ is Anosov (see Lemma \ref{lem:reduction}). With this perspective, we start by studying
real-analytic volume-preserving Anosov flows $\varphi^t \colon UM \rightarrow UM$. We say that
\begin{equation}\label{defn:totally-psi-invariant}
 \parbox{.5\textwidth}{a submanifold $N$ of $M$ is \emph{totally \mbox{$\varphi$-in\-va\-ri\-ant}} if $UN$ is a $\varphi^t$-invariant subset of $UM$ for all $\,t\in \R$.}
\end{equation}
In \eqref{defn:totally-psi-invariant}, $N$ is equipped with its induced metric so that $UN$ makes sense.
When $\varphi^t$ is either the magnetic (or, geodesic) flow, we recover the usual definition of totally magnetic (or, totally geodesic) submanifolds.
In Section~\ref{sec:dynamical-geometry}, we introduce a \emph{dynamical second fundamental form}~${\rm I\!I}^\varphi$, which precisely captures what it means for a submanifold to be totally $\varphi$-invariant (see Corollary~\ref{thm:characterization-dynamical-second-fund-form}).

We also define a \emph{dynamical exponential mapping} $\exp_x^\varphi$ in Section~\ref{sec:dynamical-geometry}. With this, we say that the flow $\varphi^t \colon UM \rightarrow UM$ is \emph{without conjugate points} if for every $x \in M$, the mapping $\exp_x^{\varphi}$ is a~local diffeomorphism.
We will consider Anosov flows which, on top of being without conjugate points, satisfy the additional transversality property:
\begin{equation} \label{eqn:condition}
 \mathbb{V}(x,v) \cap \E^{s/u}(x,v) = \{0\} \qquad\text{for all }(x,v) \in UM,
\end{equation}
where $\mathbb{V}$ is the vertical distribution of the fibration $UM\to M$ and $\mathbb{E}^{s/u}$ are the stable and unstable distributions of the flow $\varphi^t$ (see Section~\ref{subsec:anosov-flows} for more details on Anosov flows). Condition~\eqref{eqn:condition} is related to being without conjugate points for a variety of flows, and generally comes for free from the Anosov condition \cite{contreras2003asymptotic, Paternains-94-convex}. As arbitrary flows in $UM$ do not come from Hamiltonian flows on $TM$, it is not immediately clear whether all Anosov semi-spray flows on~$UM$ are without conjugate points or satisfy~\eqref{eqn:condition} (see~\cite{cuesta2025creation} for some results in this direction on surfaces).

Both the dynamical exponential map and the dynamical second fundamental form allow for us to follow the first part of the strategy in \cite[Section~3.1]{filip2024finiteness} for a special class of real-analytic volume-preserving Anosov flows.

\begin{theoremA} \label{thm:vol-preserving}
 Let $(M,\g)$ be a closed, connected, real-analytic Riemannian manifold of dimension at least $3$, and $\varphi^t\colon UM\to UM$ be a real-analytic volume-preserving Anosov flow which is without conjugate points. If it also satisfies \eqref{eqn:condition} and there are infinitely many closed totally $\varphi$-invariant hypersurfaces in $M$, then every unit tangent vector lies in a totally $\varphi$-invariant hypersurface.
\end{theoremA}

The next step in the strategy would be to show that every hyperplane tangent to $M$ can be realized by a totally $\varphi$-invariant hypersurface. The arguments in \cite{filip2024finiteness} take advantage of both the Brin group and the dynamics of the frame flow. In absence of a linear connection compatible with \emph{both} the metric and the flow, it is unclear how to define a frame flow from an arbitrary volume-preserving Anosov flow beyond the magnetic case. To that end, we restrict to the magnetic setting in the following.

\begin{theoremA} \label{thm:vol-preserving2}
 Let $(\g, \sigma)$ be a real-analytic magnetic system on a closed, connected, real-analytic manifold $M$ of dimension at least $3$ and $s > 0$. If the $s$-magnetic flow is Anosov and there are infinitely many closed totally $s$-magnetic hypersurfaces in $M$, then \emph{every} hyperplane tangent to~$M$ is realized as the tangent hyperplane to a totally $s$-magnetic hypersurface.
\end{theoremA}

To conclude the proof of Theorem \ref{thm:magnetic-finiteness}, we establish a dynamical version of \emph{Cartan's axiom of \mbox{$k$-planes}} \cite[Theorem 1.16]{dajczer-tojeiro}, which is classically characterized as follows:\ if $(M,\g)$ is a Riemannian manifold with dimension $n \geq 3$ and the property that, for some $1<k<n$, every $k$-plane tangent to $M$ is realized as the tangent space to a $k$-dimensional totally geodesic submanifold, then $(M,\g)$ has constant sectional curvature. We generalize this by replacing totally geodesic submanifolds with totally \mbox{$\varphi$-in\-va\-ri\-ant} submanifolds in the setting where the flow is \emph{odd} and \emph{without fixed points in $M$}, meaning the horizontal and vertical components of its infinitesimal generator are odd, and the former never vanishes (see \eqref{odd-defn} and the discussion after).

\begin{theoremA}\label{dyn-Cartan}
 Let $(M,\g)$ be a Riemannian manifold of dimension $n \geq 3$, $\varphi^t\colon UM\to UM$ be odd and without fixed points in $M$, and $1<k<n$ be an integer. If every $k$-plane tangent to $M$ is realized as the tangent space to a $k$-dimensional totally $\varphi$-invariant submanifold, then $(M,\g)$ has constant sectional curvature and $\varphi^t$ is a smooth time-change of the geodesic flow. Moreover, if $\varphi^t$ is a semi-spray flow, then it is the geodesic flow of $\g$.
\end{theoremA}

In \cite[Section 4.2]{filip2024finiteness}, the authors discuss whether an Anosov flow admitting infinitely many flow-invariant submanifolds with dimension at least two is necessarily algebraically defined. As~pointed out to us by Fisher in private communication, one can show that this is not the case for Anosov flows on solvmanifolds. Either by perturbing an algebraic flow or by looking at suspensions of products of Anosov toral automorphisms with Anosov diffeomorphisms that are not algebraic, it is not hard to produce examples where there are infinitely many flow-invariant submanifolds of various dimensions, but the flow is not smoothly conjugate to an algebraic flow.

Theorem \ref{thm:magnetic-finiteness} gives some evidence that such rigidity may hold once one leaves the realm of solvmanifolds; in particular, one can widen the pool from the Riemannian setting to the magnetic setting. It is an interesting question whether one could further improve the result to all odd semi-spray flows on $UM$. While many of the arguments carry over, the main issue lies within the proof of Theorem \ref{thm:vol-preserving2}; one would have to find an appropriate isometric extension of the flow (see Remark \ref{bugs} for more details).

{\bf Organization of the paper.} In Section~\ref{sec:preliminaries}, we gather some preliminaries about volume-preserving Anosov flows, principal isometric extensions of such flows, and their Brin groups. Section~\ref{sec:dynamical-geometry} introduces the notions of the dynamical second fundamental form and dynamical exponential map. Here, we also prove Theorem \ref{dyn-Cartan}. Finally, in Section~\ref{sec:rigidity}, we outline and adapt the arguments given in \cite{filip2024finiteness} in order to establish Theorems~\ref{thm:vol-preserving} and~\ref{thm:vol-preserving2}; Theorem \ref{thm:magnetic-finiteness} follows.

\section{Preliminaries}\label{sec:preliminaries}

In this section, we review some basics on the dynamics of Anosov flows as well as the Brin group associated with an isometric extension of a volume-preserving Anosov flow.

\subsection{Volume-preserving Anosov flows} \label{subsec:anosov-flows}

We follow \cite[Section 2.1]{filip2024finiteness} as well as \cite{fisher}. Let $N$ be a closed, connected Riemannian manifold and let $\|\cdot\|$ denote the norm induced by the metric. A smooth flow $\varphi^t \colon N \rightarrow N$ is \emph{Anosov} if there exists a continuous \mbox{${\rm d}\varphi^t$-in\-va\-ri\-ant} splitting $TN = \E^s \oplus \E^c \oplus \E^u$ and constants $C, \lambda > 0$ such that $\E^c$ is a one-dimensional distribution tangent to the flow and for all $t \geq 0$, $x \in N$, and $v^{s/u} \in \E^{s/u}(x)$, we have
\begin{equation}\label{Anosov-bounds} \bigl\|{\rm d}_x \varphi^t(v^s)\bigr\|_{\varphi^t(x)} \leq C {\rm e}^{-\lambda t} \|v^s\|_x \qquad \text{and}\qquad \bigl\|{\rm d}_x\varphi^{-t}(v^u)\bigr\|_{\varphi^{-t}(x)} \leq C {\rm e}^{-\lambda t} \|v^u\|_x. \end{equation}
Note that since $N$ is compact,
\eqref{Anosov-bounds}
is independent of the choice of metric, after adjusting $C$ and $\lambda$ if needed.
We say that $\varphi^t$ is \emph{volume-preserving} if there is a nowhere-vanishing top-degree form $\omega$ which is invariant under the flow. If the flow is volume-preserving, then we will denote by $\mu$ the volume measure associated to the invariant form.

\begin{Example} Let $M$ be a smooth closed manifold of dimension $n$ and let $(\g, \sigma)$ be a magnetic system on $M$. The tautological $1$-form $\alpha$ on $TM$ associated with the metric $\g$ is given by \mbox{$\alpha_{(x,v)}(\xi) = \g_x({\rm d}_{(x,v)}\pi(\xi), v)$}. For any $s>0$, the restriction of $\alpha$ to $\Sigma_s$ (cf.\ \eqref{defn-Sigma_s}) is a contact form \cite{paternain1999geodesic}, allowing us to define the \emph{Liouville volume form} $\omega$ on $\Sigma_s$ by $\omega = \alpha \wedge ({\rm d}\alpha)^{\wedge 2n}\neq 0$. It is well known that, for any $\sigma$, the $s$-magnetic flow of $(\g,\sigma)$ preserves $\omega$.
\end{Example}

The flow $\varphi^t$ gives rise to the \emph{stable and unstable foliations} on $N$, which we denote by $W^{s/u}$. Their leaves are smooth immersed submanifolds of $N$, called \emph{stable and unstable manifolds}, respectively. In particular, it holds that $T_xW^s(x) = \mathbb{E}^s(x)$ and $T_xW^u(x) = \mathbb{E}^u(x)$ for each $x\in N$. Furthermore,
\begin{equation}\label{local-s/u}
 \parbox{.66\textwidth}{there is a uniform $\varepsilon > 0$ such that the \emph{local stable and unstable manifolds} ${W^{s/u}_{\rm loc}(x) = W^{s/u}(x) \cap B(x, \varepsilon)}$ are submanifolds of $N$.}
\end{equation}
A defining property of the stable and unstable manifolds is that
\begin{equation}\label{Ws-bounded-distance} \begin{split}
&W^s(x) = \bigl\{y \in N \mid d\bigl(\varphi^t(x), \varphi^t(y)\bigr) \xrightarrow[]{t \rightarrow \infty} 0\bigr\},
 \\
& W^u(x) = \bigl\{ y \in N \mid d\bigl(\varphi^{-t}(x), \varphi^{-t}(y)\bigr) \xrightarrow[]{t \rightarrow \infty} 0\bigr\}, \end{split}\end{equation}where $d$ is
 the distance function induced by the metric on $N$. We also define the \emph{center-stable and center-unstable manifolds} by
\begin{equation}\label{center-stable} W^{cs}(x) = \bigcup_{t \in \mathbb{R}} W^s\bigl(\varphi^t(x)\bigr) \qquad\text{and}\qquad W^{cu}(x) = \bigcup_{t \in \mathbb{R}} W^u\bigl(\varphi^t(x)\bigr),\end{equation}and call a subset $X \subseteq N$ \emph{$\eta$-saturated} if $W^\eta(x) \subseteq X$ for every $x \in X$, where \mbox{$\eta \in \{s, u, cs, cu\}$.}

As in \eqref{local-s/u}, we may also consider $W^{cs}_{\rm loc}(x)$ and $W^{cu}_{\rm loc}(x)$. Given a second point $y\in N$, we define their \emph{Bowen bracket} to be $[x,y] = W^u_{\rm loc}(x) \cap W^{cs}_{\rm loc}(y)$. As $W^{cs}$ and $W^u$ are everywhere transverse, there is a neighborhood of $U$ for which $[x,y]$ consists of a single point $z$ whenever $y\in U$, and \begin{equation}\label{pi-cs-u}
\parbox{.85\textwidth}{the resulting map $\pi^{cs}_u$ on $U$, given by $\pi^{cs}_u(y)=z$, is continuous.}
\end{equation}
The classical Hopf argument yields that all volume-preserving Anosov flows are ergodic with respect to their volume measure $\mu$ \cite[Lemma A.1]{lefeuvre2023isometric}. The forward- and backward-orbits of a~point $x \in N$ are respectively denoted by $\mathcal{O}^+(x) = \bigl\{ \varphi^t(x) \mid t \geq 0\bigr\}$ and $\mathcal{O}^-(x)= \bigl\{\varphi^t(x) \mid t \leq 0\bigr\}$.
We say that $x \in N$ is \emph{forward-generic} (or, \emph{backward-generic}) if $\mathcal{O}^+(x)$ (or, $\mathcal{O}^-(x)$) is dense in $N$; it is \emph{generic} if it is either forward- or backward-generic, and \emph{two-sided generic} if the full orbit $\mathcal{O}(x) = \mathcal{O}^-(x) \cup \mathcal{O}^+(x)$ is dense in $M$.

As a consequence of Birkhoff's ergodic theorem, the subset $D\subseteq N$ consisting of all two-sided generic points has full $\mu$-measure. Furthermore, the subsets $D^+, D^-\subseteq N$ consisting of all forward- and backward-generic points are dense in $N$, with $D^+$ and $D^-$ being $s$-saturated and $u$-saturated, respectively. As $D^\pm$ are invariant under the flow, we also conclude that they are $cs/cu$-saturated sets of full $\mu$-measure. We record the following lemma, to be used in Section~\ref{sec:rigidity}.

\begin{Lemma}[{\cite[Corollary 2.1.6]{filip2024finiteness}}] \label{lem:cor-2-1-6}
 Let $\varphi^t \colon N \rightarrow N$ be a volume-preserving Anosov flow. For any $x \in N$ and any open subset $V \subseteq W^u_{\rm loc}(x)$ $($resp.\ $V \subseteq W^s_{\rm loc}(x))$, we have $V \cap D^+\neq \varnothing$ $($resp.\ $V \cap D^- \neq \varnothing)$.
\end{Lemma}

\subsection{Principal isometric extensions}

We follow \cite[Section 2]{lefeuvre2023isometric}. Let $\varphi^t \colon N \rightarrow N$ be a volume-preserving Anosov flow. By a \emph{Riemannian fiber bundle} over $N$ with typical fiber a closed Riemannian manifold $(F,\g_F)$, we mean a smooth fiber bundle $p \colon E \rightarrow N$ with typical fiber $F$ which admits a reduction of its structure group to ${\rm Iso}(F,\g_F)$. In particular, this implies that for each $x \in N$, the fiber $E_x$ has a Riemannian metric which makes it isometric to $(F,\g_F)$. A flow $\Phi^t\colon E\to E$ is called an \emph{isometric extension} of $\varphi^t$ if, for each $t\in \R$ and $x\in N$, we have that $p\circ \Phi^t = \varphi^t\circ p$ and $\Phi^t\colon E_x\to E_{\varphi^t(x)}$ is an isometry. If~$p\colon E\to N$ also happens to be a principal bundle, we call $\Phi^t$ a \emph{principal isometric extension} of~$\varphi^t$.

The principal isometric extension $\Phi^t$ of a volume-preserving Anosov flow $\varphi^t$ is a \emph{partially hyperbolic flow} on $E$. Namely, we have a direct-sum splitting $TE = \E^{s,E} \oplus \smash{\bigl(\mathbb{V} \oplus \E^{c,E}\bigr)} \oplus \E^{u,E}$, where $\mathbb{V}$ is the vertical distribution of the fibration $p\colon E\to N$, the distribution $\E^{c,E}$ is one-dimensional and tangent to the flow $\Phi^t$, and the stable and unstable bundles $\E^{s/u,E}$ of $\Phi^t$ on~$E$ are lifts of the stable and unstable bundles $\E^{s/u}$ of $\varphi^t$ on $N$ \smash{\big(i.e., ${\rm d}p\bigl(\E^{s/u,E}\bigr) = \E^{s/u}$\big)}. As in Section~\ref{subsec:anosov-flows}, these bundles are all uniquely integrable, and we denote the corresponding stable and unstable foliations on $E$ by $W^{s/u,E}$. In particular, with \eqref{local-s/u} and suggestive notation, we may also consider local stable and unstable manifolds on the level of $E$.

Given $x\in N$ and $y\in W^{s/u}_{\rm loc}(x)$ it holds that, for each $v\in E_x$, the intersection $W^{s/u,E}_{\rm loc}(v)\cap E_y$ consists of a single point $w$, characterized by the property that $d(\Phi^t(v), \Phi^t(w))\to 0$ as $t \to +\infty$. This allows us to consider
\begin{equation}\label{holonomy-map}
 \parbox{.9\textwidth}{the \emph{holonomy mapping} $H_{x,y}^{s/u, E}\colon E_x\to E_y$, given by $H_{x,y}^{s/u,E}(v) = w$,}
\end{equation}which is continuous and may be naturally extended to the entire stable/unstable leaf. We~may also consider holonomy mappings built with the local center-stable and center-unstable manifolds instead.

For any closed subset $X\subseteq E$ and generic point $x\in D^+\cup D^-\subseteq N$, we let the \emph{generic part} of $X$ be defined as
\begin{equation}\label{generic-part}
 X^{\rm gen}_x = \overline{\bigcup_{t\in \R} \Phi^t(X_x)},\qquad \mbox{where } X_x = X \cap E_x.
\end{equation}In \cite[Proposition 2.2.10]{filip2024finiteness}, it is established that if $X$ is $\Phi^t$-invariant, then $X^{\rm gen}_x$ is independent of the choice of generic point $x$---in the sense of the paragraph after \eqref{pi-cs-u}---and invariant under the holonomy mappings \eqref{holonomy-map}. Hereafter, we denote $X^{\rm gen}_x$ simply by $X^{\rm gen}$.

We may now consider \begin{equation}\label{defn-Gamma}
 \parbox{.59\textwidth}{the collection $\Gamma$ of all closed and $\Phi^t$-invariant subsets of $E$ which are mapped surjectively onto $N$ under $p\colon E\to N$.}
\end{equation} An element $X\in \Gamma$ is called \emph{relatively minimal} if it is a minimal element of $\Gamma$ relative to the inclusion order, and we let $\Gamma_{\min}\subseteq \Gamma$ denote the subcollection of all relatively minimal subsets. The reason for this terminology is that in the case where $p\colon E\to N$ is a principal $G$-bundle, with $G$ being a compact Lie group, it holds that $\Gamma_{\min}\neq \varnothing$ and $X = X^{\rm gen}$ for all $X\in \Gamma_{\min}$. Furthermore, if $x$ is generic, any $v\in X_x$ has a dense orbit in $X$ (see \eqref{generic-part}).

Finally, still in the situation where $p\colon E\to N$ is a principal $G$-bundle, it follows from \cite[Propositions 2.2.16 and 2.2.18 and Corollary 2.2.17]{filip2024finiteness} that relatively minimal sets give rise to reductions of the structure group of $E$. Namely, if $X\in \Gamma_{\min}$,
\begin{equation}\label{HX}
 \parbox{.57\textwidth}{there is a closed subgroup $H_X\leq G$ which acts freely and transitively on the fibers $X_x$, for any point $x\in N$,}
\end{equation}and a continuous reduction of the structure group to $H_X$. The subgroup $H_{X^{\rm gen}}$, given by \eqref{HX} applied to $X^{\rm gen}$ instead of $X$, is called the \emph{transitivity group} of $X$ and a priori depends on $X$. We will see that they are all isomorphic to each other, and they are also isomorphic to the Brin group of $p\colon E\to N$ discussed below.

\subsection{Brin group of isometric extensions}

Let $\Phi^t \colon E \rightarrow E$ be an isometric extension of a volume-preserving Anosov flow $\varphi^t \colon N \rightarrow N$. For any $s/u$-curve---by which we mean a curve $\gamma \colon [0,T] \rightarrow N$ whose image lies entirely in a stable or unstable leaf---we may consider the holonomy mapping \smash{$H_\gamma^E = H^{s/u,E}_{x,y}$} along $\gamma$, cf.\ \eqref{holonomy-map}, where $x=\gamma(0)$ and $y=\gamma(T)$; it is clearly invariant under reparametrizations of $\gamma$. More generally, if $\gamma = \gamma_1 \ast \cdots \ast\gamma_r$ is an $su$-curve, that is, the concatenation of $s/u$-curves, we may define the \emph{holonomy cycle mapping} of $\gamma$ by $H^E_\gamma = H^E_{\gamma_r}\circ \cdots \circ H^E_{\gamma_1}$. The \emph{holonomy group} $\mathcal{H}_x$ based at a point $x\in N$ is then defined to be the group of all holonomy cycle mappings associated with $su$-loops based at $x$. The conjugacy class of $\mathcal{H}_x$ is, as usual, independent of $x$. Accessibility of the flow easily implies that \begin{equation}\label{conjugacy-class-hol}
 \parbox{.47\textwidth}{for any points $x,y\in N$ there is an $su$-curve $\eta$ from $x$ to $y$, so that $\mathcal{H}_y = H^E_\eta \circ \mathcal{H}_x\circ \bigl(H^E_\eta\bigr)^{-1}$.}
\end{equation}
Now assume that $p\colon E\to N$ is a principal $G$-bundle, where $G$ is a compact Lie group, and that $\Phi^t$ is a principal isometric extension of $\varphi^t$. For any $x\in N$, there is a (non-canonical) isomorphism $\theta_x\colon {\rm Aut}(E_x)^G \to G$, where ${\rm Aut}(E_x)^G$ denotes the group of all $G$-equivariant self-mappings of $E_x$. While $\mathcal{H}_x$ is a subgroup of ${\rm Aut}(E_x)^G$ by design, \eqref{conjugacy-class-hol} tells us that the image $\theta_x(\mathcal{H}_x)$ in $G$ is independent of the choice of~$x$. The \emph{Brin group} is finally defined to be the closure $\mathcal{B} = \overline{\theta_x(\mathcal{H}_x)}$ in $G$.

From \cite[Corollary 2.2.13 and Propositions~2.2.15 and~2.2.16]{filip2024finiteness}, we obtain some of the main properties of the Brin group: (i) $X^{\rm gen}$ is $\mathcal{B}$-invariant for each $X\in \Gamma$, cf.~\eqref{defn-Gamma}, (ii) $X_x$ in \eqref{generic-part} is a single $\mathcal{B}$-orbit for each $x\in N$ and $X\in \Gamma_{\min}$, and (iii) for any $X\in \Gamma_{\min}$, the subgroup $H_X$ in~\eqref{HX} is in fact the Brin group $\mathcal{B}$. In particular, we see that the transitivity group of $X \in \Gamma_{\min}$ is independent of the choice of $X$, and is simply the Brin group.

The last fact we will need about principal isometric extensions in Section~\ref{sec:rigidity} is that
\begin{equation}\label{ergodic-component}
 \parbox{.5\textwidth}{there is a $\Phi^t$-invariant principal $\mathcal{B}$-subbundle $E_{\mathcal{B}}$ of $E$ over which the restriction of $\Phi^t$ is ergodic.}
\end{equation}For a proof, see \cite[Proposition 3.2]{cekic2024ergodicity}.

\section{Dynamical geometry} \label{sec:dynamical-geometry}

Throughout this section, let $(M,\g)$ be an $n$-dimensional Riemannian manifold, \mbox{$\pi\colon UM\to M$} be its unit-tangent bundle, and $\varphi^t\colon UM \to UM$ be a flow. The aim of this section is to outline all of the dynamical tools we need based on the arguments in \cite{filip2024finiteness}.

\subsection{Dynamical second fundamental form}
Let $N$ be a submanifold of $M$ equipped with its induced metric. Recall that $N$ is \emph{totally \mbox{$\varphi$-invariant}} if $UN$ is a $\varphi^t$-invariant subset of $UM$ for all $t \in \mathbb{R}$, cf.\ \eqref{defn:totally-psi-invariant}. In order to characterize this condition, we first need a description of the tangent bundle $TUN$ as a vector subbundle of $TUM|_N$. The two ingredients for that are the second fundamental form ${\rm I\!I}$ of $N$, and the Levi-Civita connector $K\colon TUM \to E$ (here, $E$ is as in the lines preceding \eqref{R-and-A}).

Below, we use the direct-sum decomposition $TM|_N = TN\oplus [TN]^\perp$ induced by $\g$, while $[\,\cdot\,]^\top\colon TM|_N\to TN$ and $[\,\cdot\,]^\perp\colon TM|_N \to [TN]^\perp$ denote the projections onto the tangent and normal bundles of $N$.

\begin{Lemma}\label{lem:characterization}
 For all $(x,v) \in UN$,
 \begin{gather} \label{TUN-description} T_{(x,v)}UN
 = \bigl\{ \xi\! \in\! T_{(x,v)}UM \mid [{\rm d}_{(x,v)}\pi(\xi)]^\perp\! =\! 0  \ \text{and} \  {\rm I\!I}_x\bigl([{\rm d}_{(x,v)}\pi(\xi)]^\top, v\bigr)\!=\! [K_{(x,v)}(\xi)]^\perp \bigr\}.\!\!\!
 \end{gather}
\end{Lemma}

\begin{proof}
 Let $\xi\in T_{(x,v)}UN$ be the initial velocity of a curve $t\mapsto (x(t),v(t))$ in $UN$, so that \mbox{$\dot{x}(0) = {\rm d}_{(x,v)}\pi(\xi)$} and $({\rm D}v/{\rm d}t)(0) = K_{(x,v)}(\xi)$, and let $\eta$ be any local section of $[TN]^\perp$ defined around $x$. The condition $[{\rm d}_{(x,v)}\pi(\xi)]^\perp= 0$ is directly obtained from evaluating the equality $\g_{x(t)}(\dot{x}(t), \eta_{x(t)}) = 0$ at $t=0$. We may then differentiate $\g_{x(t)}(v(t), \eta_{x(t)})=0$ at $t=0$ to obtain \begin{equation}\label{diff-normal-section}\g_x(K_{(x,v)}(\xi), \eta_x) + \g_x(v, \nabla_{{\rm d}_{(x,v)}\pi(\xi)}\eta)=0,\end{equation}where $\nabla$ is the Levi-Civita connection of $(M,\g)$. Letting $A\colon T_xN\to T_xN$ denote the shape operator of $\eta_x$, we have that \begin{equation*}
 \g_x(v, \nabla_{{\rm d}_{(x,v)}\pi(\xi)}\eta) = -\g_x(v, A({\rm d}_{(x,v)}\pi(\xi))) = -\g_x( {\rm I\!I}_x({\rm d}_{(x,v)}\pi(\xi) ,v), \eta_x),
 \end{equation*}so that \eqref{diff-normal-section} reduces to $\g_x(K_{(x,v)}(\xi) - {\rm I\!I}_x({\rm d}_{(x,v)}\pi(\xi) ,v),\eta_x ) = 0$. The arbitrariety of $\eta$ now implies that ${\rm I\!I}_x({\rm d}_{(x,v)}\pi(\xi), v) = [K_{(x,v)}(\xi)]^\perp$.

 The conclusion follows from a dimension count: if $k=\dim N$, the right side of \eqref{TUN-description} has the dimension $2n-1-2(n-k) = 2k-1$, as does $T_{(x,v)}UN$. \end{proof}

Recall that $K$ also yields a direct-sum decomposition $T_{(x,v)}UM \cong \mathbb{H}(x,v) \oplus \mathbb{V}(x,v)$, where $\mathbb{H}(x,v) = \ker(K_{(x,v)})$ and $\mathbb{V}(x,v) = \ker({\rm d}_{(x,v)}\pi)$ are called the \emph{horizontal} and \emph{vertical} bundles, and both restrictions
\begin{equation}\label{dpi-K-isomorphisms}
{\rm d}_{(x,v)}\pi \colon \ \mathbb{H}(x,v) \rightarrow T_xM\qquad\mbox{and}\qquad K_{(x,v)} \colon \ \mathbb{V}(x,v) \rightarrow v^\perp
\end{equation}are isomorphisms.

The infinitesimal generator $X$ of $\varphi^t$ is completely determined by its horizontal and vertical components, respectively defined by \begin{equation}\label{hor-ver-comps} X_H(x,v) = {\rm d}_{(x,v)}\pi(X(x,v))\qquad\mbox{and} \qquad X_V(x,v) = K_{(x,v)}(X(x,v)).\end{equation} With this in place, we say that $\varphi^t$ is \emph{odd} with respect to $\g$ if both of its components are odd, that is, if
\begin{equation}\label{odd-defn} X_H(x,-v) = - X_H(x,v) \qquad\text{and}\qquad X_V(x,-v) = -X_V(x,v)\end{equation}for all $(x,v)\in UM$. Note that the metric is crucial for this definition, since \eqref{odd-defn} cannot be replaced by the single condition $X(x,-v) = -X(x,v)$, as $T_{(x,-v)}UM \neq T_{(x,v)}UM$.

Furthermore, we say that $x\in M$ is a \emph{fixed point} of $\phi^t$ if there exists $v \in U_xM$ such that $\pi \circ \varphi^t(x,v) = x$ for all $t \in \mathbb{R}$. Equivalently, $\varphi^t$ has a fixed point in $M$ if there exists $(x,v) \in UM$ such that $X_H(x,v) = 0$. Finally, $\varphi^t$ is called a \emph{semi-spray flow} if it satisfies that $X_H(x,v) = v$. Note that they are always without fixed points.

\begin{Example}\label{example:odd-flows}
 As in \eqref{eqn:magnetic-geodesic}, the orbits of a semi-spray flow \mbox{$\varphi^t \colon UM\to UM$} are also described by a second-order differential equation. Namely, the condition $X_H(x,v) = v$ says that all such orbits $t\mapsto (\gamma(t),v(t))$ must have $v(t) = \dot{\gamma}(t)$, while taking covariant derivatives of the vector part of $(\gamma(t),\dot{\gamma}(t)) = \varphi^t(\gamma(0),\dot{\gamma}(0))$ leads to
 \begin{equation}\label{Dc'/dt=XV}
 \frac{{\rm D}\dot{\gamma}}{{\rm d}t}(t)= X_V(\gamma(t),\dot{\gamma}(t)).
 \end{equation}
 In particular, when $\varphi^t$ is the geodesic flow of $\g$, we have that $X_H(x,v)=v$ and \mbox{$X_V(x,v)=0$}. When $\varphi^t$ is the magnetic flow of a magnetic system $(\g,\sigma)$, we have that $X_H(x,v)=v$ and \mbox{$X_V(x,v)=\Y_x(v)$}, for $\Y$ as in \eqref{Lorentz-force}. Both such semi-spray flows are odd. For more examples on closed surfaces, we refer the reader to \cite{cuesta2025thermostats}.
\end{Example}

Inspired by Lemma \ref{lem:characterization}, we have the following.

\begin{Definition}\label{dyn-2FF}
 The \emph{dynamical second fundamental form} of $N$ relative to $\g$ and the flow $\varphi^t\colon UM\to UM$ is the mapping ${\rm I\!I}^{\varphi}\colon UN\to [TN]^\perp \oplus [TN]^\perp$ defined by \[{\rm I\!I }^{\varphi}_x(v) = \bigl({\rm I\!I}_x\bigl([X_H(x,v)]^\top, v\bigr) - [X_V(x,v)]^\perp,
 [X_H(x,v)]^\perp\bigr)\] for all $(x,v)\in UN$.
\end{Definition}

Note that the second component of ${\rm I\!I}^\varphi$ vanishes identically whenever $\varphi^t$ is a semi-spray flow. In particular, in view of Example \ref{example:odd-flows}, this dynamical second fundamental form reduces to both the classical one in the case of the geodesic flow, and also to the magnetic second fundamental form introduced in Definition~6.9 in the first arXiv version of \cite{ABM_2025} and in \cite[Definition 4.1]{terek2025submanifold} in the case of a magnetic flow. The significance of the dynamical second fundamental form is that it precisely captures what it means for a submanifold to be totally $\varphi$-invariant.

\begin{Corollary} \label{thm:characterization-dynamical-second-fund-form}
 A submanifold $N$ of $M$ is totally $\varphi$-invariant if and only if its dynamical second fundamental form vanishes identically.
\end{Corollary}

\begin{proof}
 It suffices to note that $UN$ is $\phi^t$-invariant for all $t\in \R$ if and only if we have that $X(x,v) \in T_{(x,v)}UN$ for all $(x,v)\in UN$. Such tangency condition, due to Lemma \ref{lem:characterization} and the definitions of $X_H$ and $X_V$ in \eqref{hor-ver-comps}, means precisely that ${\rm I\!I}^\varphi = 0$.
\end{proof}

For magnetic flows, Corollary~\ref{thm:characterization-dynamical-second-fund-form} reduces to \cite[Theorem~1.4]{ABM_2025} and \cite[Remark~4.2\,(i)]{terek2025submanifold}.
We now have the tools to prove Theorem~\ref{dyn-Cartan}, which can be seen as a dynamical version of Cartan's axiom of $k$-planes.

\begin{proof}[Proof of Theorem \ref{dyn-Cartan}]
 Note that, as $1< k<n$, for any $(x,v)\in UM$, we have
 \begin{equation}\label{intersection}
 \bigcap \{ W \leqslant T_xM \mid \dim W=k\mbox{ and }v\in W\} = \R v.
 \end{equation}Now consider any $k$-dimensional subspace $W$ of $T_xN$ containing $v$, and fix a totally \mbox{$\varphi$-in\-va\-ri\-ant} submanifold $N$ of $M$ with $x\in N$ and $W=T_xN$. As $(x,v) \in UN$, Definition~\ref{dyn-2FF} and Corollary~\ref{thm:characterization-dynamical-second-fund-form} yield both relations \begin{equation}\label{dyn-Cartan-even-odd}
 {\rm (i)}~ {\rm I\!I}_x( [X_H(x,v)]^\top, v) = [X_V(x,v)]^\perp \qquad\mbox{and}\qquad {\rm (ii)}~[X_H(x,v)]^\perp = 0 .
 \end{equation}As $\varphi^t$ is odd, the left side of (\ref{dyn-Cartan-even-odd}(i)) is even while the right side is odd, and hence they vanish separately. Due to \eqref{intersection}, it follows from \mbox{(\ref{dyn-Cartan-even-odd}(ii))} and $[X_V(x,v)]^\perp = 0$ that \mbox{$X_H(x,v),X_V(x,v) \in \R v$}. At the same time, \mbox{$X_V(x,v) \in v^\perp$} in view of \eqref{dpi-K-isomorphisms}, and thus $X_V(x,v)=0$. We also obtain a~smooth function $\lambda\colon UM\to \R$ such that $X_H(x,v)= \lambda(x,v)v$, which has no zeros as $\varphi^t$ is without fixed points in $M$. This way, ${\rm I\!I}_x( [X_H(x,v)]^\top, v) = 0$ directly implies that ${\rm I\!I}=0$, meaning that every totally $\varphi$-invariant submanifold considered above is, in fact, totally geodesic. Hence, by the classical characterization of Cartan's axiom of $k$-planes, $\g$ has constant sectional curvature. Finally, we conclude that $X(x,v) = \lambda(x,v) X^\g(x,v)$, where $X^\g$ is the infinitesimal generator of the geodesic flow of $\g$, making $\varphi^t$ a smooth time-change of such geodesic flow. If~$\varphi^t$ is a semi-spray flow, then $\lambda(x,v) \equiv 1$.
\end{proof}

Theorem \ref{dyn-Cartan} can be specialized to the case of $s$-magnetic flows. In order to do so, a rescaling of the magnetic system is needed:\ $s$-magnetic flows are not defined on $UM$, but instead on $\Sigma_s$, cf.\ \eqref{defn-Sigma_s}.

\begin{Lemma} \label{lem:reduction}
Let $(\g,\sigma)$ be a magnetic system on a smooth manifold $M$, and $s>0$. Then,
\begin{enumerate}[\normalfont(i)]\itemsep=0pt
 \item the $s$-magnetic flow of $(\g,\sigma)$ equals the $1$-magnetic flow of $\bigl(s^{-2}\g,s^{-2}\sigma\bigr)$;
 \item any totally $s$-magnetic submanifold of $M$ for $(\g,\sigma)$ is also a totally \mbox{$1$-mag\-ne\-tic} submanifold of $M$ for $\bigl(s^{-2}\g,s^{-2}\sigma\bigr)$.
\end{enumerate}
\end{Lemma}

 \begin{proof}
 If we write $\sigma \sim_{\g} \Y$ to mean that $\sigma$ and $\Y$ are related under $\g$ as in \eqref{Lorentz-force}, multiplying both sides of \eqref{Lorentz-force} by $s^{-2}$ yields $s^{-2}\sigma \sim_{s^{-2}\g} \Y$. As $\Sigma_s$ for $\g$ equals the unit tangent bundle for $s^{-2}\g$ and \begin{equation}\label{same-LC} \parbox{.72\textwidth}{both $\g$ and $s^{-2}\g$ share the same Levi-Civita connection,}\end{equation} it now follows that $\bigl(s^{-2}\g,s^{-2}\sigma\bigr)$ gives rise to the \emph{same} equation \eqref{eqn:magnetic-geodesic} as $(\g,\sigma)$. Therefore, (i)~holds. For (ii) we first note that, as another consequence of \eqref{same-LC}, the second fundamental form of any submanifold of $M$ remains unchanged under constant rescalings of $\g$. At the same time, by (i), the same holds for the infinitesimal generator of the magnetic flow, and thus for the dynamical second fundamental form, cf.\ Definition~\ref{dyn-2FF}. With this in place, (ii) follows from Corollary \ref{thm:characterization-dynamical-second-fund-form}.
 \end{proof}

Applying Theorem \ref{dyn-Cartan} to $\bigl(s^{-2}\g,s^{-2}\sigma\bigr)$ instead of $(\g,\sigma)$ and noting that a magnetic flow can only equal the underlying geodesic flow when the magnetic $2$-form vanishes (see also the lines following \eqref{Dc'/dt=XV}), we finally obtain the magnetic version of Cartan's axiom of $k$-planes.

\begin{Corollary}\label{cor:mag-cartan}
 Let $M$ be an $n$-dimensional smooth manifold with $n\geq 3$, fix $1<k<n$, and let $(\g,\sigma)$ be a magnetic system on $M$. If, for some $s>0$, it holds that every $k$-plane tangent to $M$ is realized as the tangent space to a $k$-dimensional totally $s$-magnetic submanifold of $M$, then $\sigma=0$ and $\g$ has constant sectional curvature.
\end{Corollary}

\begin{rem}
 As a last consequence of the proof of Theorem \ref{dyn-Cartan}, when $\varphi^t\colon UM\to UM$ is an odd semi-spray flow, a submanifold $N$ of $M$ is totally $\varphi$-invariant if and only if it is both totally geodesic and with $X_V(x,v)\in T_xN$ for all $(x,v)\in UN$; cf.\ \cite[Example 4.3]{terek2025submanifold} for the case of a~magnetic flow.
\end{rem}

\subsection{Dynamical exponential map}

Denoting by $\|\cdot\|$ the norms induced by $\g$ on the tangent space of $M$, for each $x\in M$ the \emph{dynamical ex\-po\-nen\-ti\-al map} relative to the flow $\varphi^t$ is the mapping $\exp_x^{\varphi} \colon T_xM \rightarrow M$ given by
\begin{equation}\label{def-psi-exp} \exp_x^{\varphi}(v) = \begin{cases} \pi \circ \varphi^{\|v\|}(x,v/\|v\|) &\text{if } v \neq 0, \\ x & \text{if } v = 0. \end{cases}\end{equation}

The following lemma is a direct adaptation of \cite[Lemma A.7]{dairbekov2007boundary}. We still provide its proof for the reader's convenience (here, we use Einstein's summation convention).

\begin{Lemma} \label{lem:dynamical-exp}
 Assume that $\varphi^t$ is of class $C^k$, with $k \geq 2$ $($or $C^\infty$ or $C^\omega)$. Then for any point $x\in M$, $\exp^\varphi_x$ has the same regularity as $\varphi^t$ on $T_xM\smallsetminus \{0\}$, and it is of class $C^1$ on $T_xM$. Moreover, if it is of class $C^2$ at $0$ and $\varphi^t$ is an odd semi-spray flow, then $X_V(x,v) = 0$ for all $v\in U_xM$.
\end{Lemma}

\begin{proof}
 The $C^k$-regularity of $\exp^{\varphi}_x$ on $T_xM \smallsetminus \{0\}$ follows from \eqref{def-psi-exp}. In addition, all of its directional derivatives at $0$ exist as
 \begin{equation}\label{d0-psi-exp} {\rm d}_0\exp_x^{\varphi}(v) = \frac{\rm d}{{\rm d}t} \Big|_{t=0} \pi \circ \varphi^t(x,v) = X_H(x,v)\end{equation}for all $(x,v)\in UM$. Fixing coordinates $x^1, \ldots, x^n$ for $M$ centered at $x$, set
 \begin{equation}\label{gamma-i}
 \gamma^i(t,x,v) = \exp_x^{\varphi, i}(tv),
 \end{equation}where $\exp_x^{\varphi, i}(tv) = x^i(\exp_x^{\varphi}(tv))$; note that each $\gamma^i$ is of class $C^k$. Relative to linear coordinates $v^1,\ldots, v^n$ on $T_xM$, we may differentiate \eqref{gamma-i} with respect to $v^j$ to obtain
 \begin{equation}\label{d-gamma-i-d-vj}
 \frac{\partial \gamma^i}{\partial v^j}(t,x,v) =t \frac{\partial \exp_x^{\varphi,i}}{\partial v^j}(tv) \qquad\mbox{for all }t>0.
 \end{equation}Observe, however, that \eqref{d-gamma-i-d-vj} also holds at $t=0$, as $\gamma^i(0,x,v) = 0$, which is another consequence of \eqref{def-psi-exp} together with our choice of coordinates. With $(\partial \gamma^i/\partial t)(0,x,v) = X_H^i(x,v)$ for all $(x,v)\in UM$, where $X_H^i(x,v)$ denotes the $i$-th component of $X_H(x,v)$ in the coordinates $x^1,\ldots, x^n$ (and similarly for $X_V(x,v)$ below), this allows us to compute
 \begin{equation*}
\lim_{t\to 0^+} \frac{\partial \exp_x^{\varphi, i}}{\partial v^j}(tv) = \lim_{t\to 0^+} \frac{1}{t} \frac{\partial \gamma^i}{\partial v^j}(t,x,v) = \frac{\partial^2 \gamma^i}{\partial t \partial v^j}(0,x,v) = \frac{\partial^2 \gamma^i}{\partial v^j \partial t}(0,x,v) = \frac{\partial X_H^i}{\partial v^j}(x,v).
 \end{equation*}
The above limit is uniform in $v$ as $U_xM$ is compact and, together with \eqref{d0-psi-exp}, shows that $\exp_x^{\varphi}$ is of class $C^1$ at $0$.

Next, assume that $\varphi^t$ is an odd semi-spray flow, and that $\exp_x^{\varphi}$ is of class $C^2$ at $0$.
For any $v\in U_xM$, the curve $t\mapsto \exp_x^{\varphi}(tv)$ is a solution of \eqref{Dc'/dt=XV}, whose coordinate version reads
\begin{equation}\label{cov-coordinates}
 \frac{\partial^2\gamma^i}{\partial t^2}(t,x,v) + \Gamma_{jk}^i(\exp_x^{\varphi}(tv)) \frac{\partial\gamma^j}{\partial t}(t,x,v)\frac{\partial \gamma^k}{\partial t}(t,x,v) = X^i_V(\varphi^t(x,v)),
\end{equation}for all $i=1,\ldots, n$. Differentiating \eqref{gamma-i} with respect to $t$ twice, we obtain
\begin{equation}\label{d2gammai}
 \frac{\partial^2\gamma^i}{\partial t^2}(t,x,v) =
\frac{\partial^2 \exp_x^{\varphi,i}}{\partial v^j\partial v^k}(tv)v^jv^k.\end{equation}Using the $C^2$-regularity of $\exp_x^{\varphi}$ in order to take $t\to 0$ in \eqref{cov-coordinates}, together with \eqref{d2gammai} and the relations $(\partial \gamma^i/\partial t)(0,x,v) = v^i$ given by the semi-spray condition, it follows that
\begin{equation}\label{even-odd-again}
\left( \frac{\partial^2 \exp_x^{\varphi,i}}{\partial v^j\partial v^k}(0)+\Gamma^i_{jk}(x)\right)v^jv^k= X^i_V(x,v).
\end{equation}The left side of \eqref{even-odd-again} is manifestly even in $v$, while the right side is odd by assumption---this is only possible if both vanish individually. Therefore, $X_V(x,v)=0$ as claimed.
\end{proof}

\section{Rigidity results}\label{sec:rigidity}

The main goal of this section is to establish Theorems \ref{thm:vol-preserving} and \ref{thm:vol-preserving2}, working in the real-analytic category. We recall some of the notation from \cite{filip2024finiteness}. Let $(M,\g)$ be a closed, real-analytic Riemannian manifold of dimension $n\geq 3$, and assume that \begin{equation}\label{phi-assumptions}
 \text{$\varphi^t \colon \ UM \rightarrow UM$ is a real-analytic volume-preserving Anosov flow.}
\end{equation}
We also consider the Grassmannian bundle $HM$ of oriented hyperplanes tangent to $M$, and let \mbox{$\pi_H \colon FM \rightarrow HM$} be the unit tautological bundle over $HM$, with fibers explicitly given by $F_{(x, \Pi)}M = \{(x,\Pi,v) \mid (x,\Pi) \in HM \mbox{ and } v \in \Pi \cap U_xM\}$. Introducing a further projection $\pi_U \colon FM \rightarrow UM$, with $\pi_U(x,\Pi,v)=(x,v)$,
and denoting by $Z \subseteq HM$ the subset consisting of all hyperplanes which are tangent to a germ of a totally $\varphi$-invariant hypersurface, we set
\begin{equation}\label{I-and-A} I =\pi_H^{-1}(Z) \subseteq FM \qquad\mbox{and}\qquad A = \pi_U(I)\subseteq UM. \end{equation}
Our first claim is that the sets $I$ and $Z$ are real-analytic and $A$ is subanalytic. The method of proving this perfectly mirrors \cite[Corollary 3.2.3]{filip2024finiteness}; namely, we wish to take advantage of \cite[Lemma~3.2.2]{filip2024finiteness}, which establishes real-analyticity of the set consisting of all base points over whose fibers a real-analytic function (defined on a bundle) vanishes. Utilizing Section~\ref{sec:dynamical-geometry},
we now present the argument in detail.

\begin{Lemma}
 The set $Z$ is real-analytic. In particular, $I$ is also real-analytic, and $A$ is subanalytic.
\end{Lemma}

\begin{proof}
 For any $(x, \Pi) \in HM$, we set $S_\Pi=\exp_x^{\varphi}(\Pi)$. Note that $S_\Pi$ is a $C^1$-immersed hypersurface of $M$ but, as $\exp_x^{\varphi}$ is real-analytic away from the origin, so is $S_\Pi \smallsetminus \{x\}$. In addition, $S_\Pi$~represents the germ of a totally \mbox{$\varphi$-in\-va\-ri\-ant} hypersurface if and only if for every $y \in S_\Pi$, unit vector $v \in T_y S_\Pi$, and $t \in \R$, we have that $\pi \circ \varphi^t(y,v) \in S_\Pi$.

 Whenever $(x, \Pi, v) \in FM$, the curve $t\mapsto \pi\circ \varphi^t(x,v)$ lies entirely in $S_\Pi$ by definition of $S_\Pi$. This allows us to define the augmented exponential map
 \begin{equation}\label{aug-exp} E \colon \ FM \times (0,\infty) \rightarrow HM \qquad \mbox{by } E(x, \Pi, v, t) = (\exp^\varphi_x(tv), {\rm d}_{tv}\exp_x^{\varphi}(\Pi)).\end{equation}
 Explicitly, $E$ sends $(x, \Pi, v)$ to the corresponding tangent hyperplane to $S_\Pi$ at $\exp_x^{\varphi}(tv)$. Furthermore, since $tv$ is never the zero vector, this is a real-analytic mapping. Define the mapping $\alpha\colon HM \to \mathbb{R}$ by
 \begin{equation}\label{defn-alpha}
 \alpha(x,\Pi) = \int_{U_xS_\Pi} \|{\rm I\!I}^\varphi_x(v)\|^2\,{\rm d}\mu(v),
 \end{equation}
 where ${\rm I\!I}_x^\varphi$ is the dynamical second fundamental form of $S_\Pi$ at $x$, the codomain $[TS_\Pi]^\perp \oplus [TS_\Pi]^\perp$ of ${\rm I\!I}^\varphi$ is equipped with the product fiber metric $\g\oplus \g$, and $\mu$ denotes the Lebesgue measure on $U_xS_\Pi$ induced by $\g$. With this in place, the pullback mapping $E^*\alpha \colon FM \times (0,\infty) \rightarrow \mathbb{R}$ is also real-analytic, and Corollary \ref{thm:characterization-dynamical-second-fund-form} yields that $(x,\Pi) \in Z$ if and only if we have that the function $E^*\alpha(x, \Pi, \cdot, \cdot) \colon U_xS_\Pi \times (0,\infty) \rightarrow \mathbb{R}$ vanishes identically. This allows us to obtain the conclusion from \cite[Lemma 3.2.2]{filip2024finiteness}, taking $\pi\colon X\to Y$ therein to be~$E$ in~\eqref{aug-exp} and $\rho = E^*\alpha$ for~$\alpha$ in~\eqref{defn-alpha}.
\end{proof}

Once $Z$ is real-analytic, it follows from nearly the same argument as in \cite[Proposition~3.2.5]{filip2024finiteness} that there is at least one totally $\varphi$-invariant hypersurface through every tangent vector. The only difference is that one has to take advantage of the fact that the flow is without conjugate points, along with~\eqref{eqn:condition} and the inverse function theorem.
We sketch the details for the reader's convenience.

\begin{proof}[Proof of Theorem \ref{thm:vol-preserving}]
 We must show that $A=UM$, cf.\ \eqref{I-and-A}. As $A$ is closed and \mbox{$\varphi^t$-invariant}, it suffices to show that there exists $(x,v) \in A$ whose forward-orbit or backward-orbit is dense. First, using the assumption that $M$ has infinitely many closed, immersed, distinct totally \mbox{$\varphi$-invariant} hypersurfaces, it follows that we have $\dim(Z) \geq n$. As a consequence, we obtain that \mbox{$\dim(I) = \dim(Z) + (n-2)$}, and hence $\dim(I) \geq 2n-2$.

 Let $A^{\mathrm{o}} \subseteq A$ be the collection of smooth points, and consider the restricted projection \mbox{$\pi_U \colon I \rightarrow A$}, cf.\ \eqref{I-and-A}. There are now two possibilities: either
 \begin{enumerate}[\normalfont(i)]\itemsep=0pt
\item $\dim(A) = \dim(I)$, or \item there is a further relatively open subset $A^{\mathrm{o} \mathrm{o}} \subseteq A$ such that $\pi_U|_I$ has fibers of positive dimension over $A^{\mathrm{o} \mathrm{o}}$.
\end{enumerate} We address each case separately.

 If possibility (i) holds, suppose that $\dim(A) = \dim(I) \geq 2n-2$. If equality holds, then $\dim(UM) = 2n-1$ yields $A = UM$. Otherwise, if $\dim(A) = 2n-2$, let $(x,v) \in A^{\mathrm{o}}$ be a smooth point, and consider the quotient bundle $NUM$ over $UM$, with fibers given by ${N_{(x,v)}UM = T_{(x,v)}UM/\mathbb{R}X(x,v)}$. The hyperplane $T_{(x,v)}A/\mathbb{R} X(x,v)$ must be transverse---with\-out loss of generality---to the stable bundle $\mathbb{E}^{s}(x,v)$ (as $\varphi^t$ is Anosov, cf.\ \eqref{phi-assumptions}), in which case we have that $\dim(\mathbb{E}^{cs}(x,v) \cap T_{(x,v)}A) = \dim(W^{cs}(x,v)) - 1 = n-1$; see \eqref{center-stable}. Let $Q$ be a real-analytic submanifold of $A^{\mathrm{o}}$ such that $T_{(x,v)}UM = T_{(x,v)}Q \oplus \mathbb{E}^{cs}(x,v)$. Such $Q$ remains transverse to $W^{cs}$ in a neighborhood $U$ of $(x,v)$, and so the projection $\pi_u^{cs}(Q \cap U) \subseteq W^u_{\rm loc}(x,v)$---cf.\ \eqref{local-s/u} and \eqref{pi-cs-u}---contains a non-empty open set. By Lemma \ref{lem:cor-2-1-6}, we obtain a forward-dense point in~$A$.

If (ii) holds instead, we may fix a point $(x,v) \in A^{\mathrm{o} \mathrm{o}}$ such that the fiber $I_{(x,v)} = \pi_H^{-1}(x,v) \subseteq I$
contains a real-analytic curve $\Gamma$. We proceed in a small neighborhood $U \subseteq M$ of $x$. For a point $i = (x, \Pi, v)$ along the curve $\Gamma$, let $H_i \subseteq U$ be a totally $\varphi$-invariant hypersurface with $T_xH_i = \Pi$. As $\exp_x^\varphi$ is a local diffeomorphism by assumption, \begin{equation}\label{U1-defn}\parbox{.82\textwidth}{the union of all such $H_i$ must contain an open subset $U_1$ of $U$.}\end{equation} By \eqref{eqn:condition}, reducing $W^{cs}_{\rm loc}(x,v)$ further if needed, the projection $UM\to M$ restricts to a diffeomorphism $\pi\colon W^{cs}_{\rm loc}(x,v)\to M$ onto its image; we denote its inverse by $\delta_{(x,v)}\colon U \to W^{cs}_{\rm loc}(x,v)$. Given $y\in U$, \begin{equation}\label{char-delta}\parbox{.65\textwidth}{the element $\delta_{(x,v)}(y)$ is characterized by the property that the distance $d(\varphi^t(\delta_{(x,v)}(y)),\varphi^t(x,v))$ remains bounded as $t\to +\infty$,}\end{equation} cf.\ \eqref{Ws-bounded-distance}. It remains to show that \begin{equation}\label{U1-image}
 \parbox{.63\textwidth}{$\delta_{(x,v)}(U_1) \subseteq A$, where $U_1$ is the open set in \eqref{U1-defn}.}
 \end{equation}
 This way, as $A \cap W^{cs}(x)$ contains an open set, we may use Lemma \ref{lem:cor-2-1-6} to deduce that $A = UM$, concluding the proof. We establish \eqref{U1-image} directly: if $p_1 \in U_1 \cap H_i$ for some $i\in \Gamma$, using that $H_i$ is a totally $\varphi$-invariant hypersurface containing both $p_1$ and $x$, we obtain a unique $v_1\in U_{p_1}H_i$ such that $d(\varphi^t(p_1,v_1), \varphi^t(x,v))$ remains bounded as $t\to +\infty$; it follows from \eqref{char-delta} that $(p_1,v_1) = \delta_{(x,v)}(p_1)$ is tangent to $H_i$, and hence in $A$ as required.
 \end{proof}

We now consider the \emph{frame bundle} $PM\to UM$, whose fiber over any element $(x,v) \in UM$ consists of all ordered orthonormal bases of $T_xM$ starting with $v$; it is easy to see that it is a~principal $\SO(n-1)$-bundle over $UM$. The construction of the frame extension of $\varphi^t$ used in~\cite{filip2024finiteness} must be modified in order for us to proceed and establish Theorem \ref{thm:vol-preserving2}. It is from this point onwards that the discussion is specialized to $s$-magnetic flows; Lemma \ref{lem:reduction} allows us to assume without loss of generality that $s=1$ in Theorems \ref{thm:magnetic-finiteness} and \ref{thm:vol-preserving2}. Along each unit-speed magnetic geodesic $\gamma\colon \R\to M$, we may consider the \emph{magnetic covariant derivative operator} $\mathcal{D}$ given by
\begin{equation}\label{mag-covariant}
 (\mathcal{D}W)(t) = \frac{{\rm D}W}{{\rm d}t}(t) - \Y_{\gamma(t)}(W(t)),
\end{equation}for any vector field $W$ along $\gamma$ (see \cite{assenza-imrn,grognet-etds} for implicit uses of $\mathcal{D}$). In fact, $\mathcal{D}$ defines a metric-compatible connection in the pullback bundle $\gamma^*(TM)$ over $\R$ (as $\Y$ in \eqref{Lorentz-force} is skew-adjoint), while \eqref{eqn:magnetic-geodesic} reduces to $\mathcal{D}\dot{\gamma}=0$.

We define the \emph{magnetic frame-extension flow} $\Phi^t\colon PM\to PM$ of $\varphi^t$ as
 \begin{equation}\label{frame-extension}\Phi^t(x, v, v_2, \ldots, v_n) = (\gamma(t), v(t), v_2(t), \ldots, v_n(t)), \end{equation}
 where we set $\varphi^t(x,v) = (\gamma(t), v(t))$ and each $v_j(t)$ is the $\mathcal{D}$-parallel transport of the corresponding vector $v_j$ along the curve $x(t)$. By metric-compatibility of $\mathcal{D}$, \eqref{frame-extension} defines a principal isometric extension of $\varphi^t$, and so we may consider its Brin group $\mathcal{B} \leq \SO(n-1)$.

 In \cite[Remark 3.3.3]{filip2024finiteness}, the authors observe that there are topological obstructions to the principal bundle having a continuous reduction of the structure group to $H \subseteq\SO(n-1)$ with~$H$ connected. Provided $n$ is odd and $n \neq 7$, there does not exist a non-trivial reduction. If $n$~is even and is not either $8$ or $134$, then there are two cases: either $n \equiv 2 \pmod{4}$ or $n=4$, in which case $H$ must fix a vector in $\mathbb{R}^{n-1}$, or $n \equiv 0 \pmod{4}$, in which case $H$ must act reducibly on $\mathbb{R}^{n-1}$. In the exceptional cases, if $n = 7$ then we must have that $H$ contains $\text{SU}(3)$ and fixes a complex structure on $\mathbb{R}^6$, if $n=8$ then either $H$ acts reducibly on $\mathbb{R}^7$ or $H$ fixes a nonzero $3$-form on $\mathbb{R}^7$ and is either $\SO(3)$ or ${\rm G}_2$, and finally if $n =134$ then $H$ either fixes a vector in $\mathbb{R}^{133}$ or it fixes a nonzero $3$-form and equals $E_7/(\mathbb{Z}/2 \mathbb{Z})$.

 With the above discussion and \cite[Step 1 of proof of Theorem 3.1]{cekic2024ergodicity}, the authors were able to leverage these obstructions in \cite[Proposition 3.3.5]{filip2024finiteness} along with the fact that there are infinitely many closed totally geodesic hypersurfaces in order to enumerate all possibilities for what \begin{equation}\label{coset-space}\mbox{the coset space }\mathcal{C}= \mathcal{B} \backslash (\SO(n-1)/\SO(n-2))\end{equation} could be, provided the Brin group $\mathcal{B}$ is connected. In light of Section~\ref{sec:dynamical-geometry}, we observe that these arguments generalize readily without any modification, provided there is a notion of a frame flow. In our setting, we have \begin{equation}\label{Brin-dichotomy}
\parbox{.75\textwidth}{if $\mathcal{B}$ is connected and $M$ admits infinitely many closed totally $1$-magnetic hypersurfaces, then either $\mathcal{B} = \SO(n-1)$ or $\mathcal{C}$ equals $[-1,1]$ or a point.}
 \end{equation}
We now explain why we may assume the Brin group is connected. Below, by a \emph{magnetomorphism} between manifolds equipped with magnetic systems, we mean a diffeomorphism which preserves both the metric and magnetic $2$-form, cf.\ \cite[Section 1.1]{ABM_2025}. Similarly, a \emph{magnetomorphic covering} between manifolds equipped with magnetic systems is a covering map which is a local magnetomorphism. The next lemma is a direct magnetic adaptation of \cite[Lemma 3.3]{cekic2024ergodicity}, whose details
we again provide.

\begin{Lemma}\label{Brin-connected}
 With the above setup, and with $\mathcal{B}^0$ denoting the identity component of $\mathcal{B}$, there is a finite magnetomorphic covering of $M$ whose magnetic frame-extension flow has the Brin group $\mathcal{B}^0$.
\end{Lemma}

\begin{proof}
 By \eqref{ergodic-component}, we may fix a $\Phi^t$-invariant principal $\mathcal{B}$-subbundle $P_{\mathcal{B}}$ of $PM$ on which $\Phi^t$ is ergodic. Since $\Phi^t$ is transitive on $P_{\mathcal{B}}$, we have that $P_{\mathcal{B}}$ is connected. As the quotient $\mathcal{B}/\mathcal{B}^0$ is discrete \cite[Proposition 6.15]{MR4242835}, the quotient projection $P_{\mathcal{B}} / \mathcal{B}^0 \to UM$ is a finite covering map, with deck transformation group $\mathcal{B}/\mathcal{B}^0$. At the same time, as $n\geq 3$, it follows from the long exact sequence in homotopy that $\pi_1(UM)\cong \pi_1(M)$. Thus, we may fix a subgroup $\Gamma \leq \pi_1(M)$ acting freely and properly discontinuously on the universal covering $\widetilde{M}$ of $M$ for which the quotient $\hat{M} = \widetilde{M}/\Gamma$, equipped with the pullback magnetic system, has its own unit-tangent bundle $U\hat{M}$ equivariantly diffeomorphic to $P_{\mathcal{B}}/\mathcal{B}^0$. Here, $U\hat{M}$ is equipped with its own $1$-magnetic flow, while $P_{\mathcal{B}}/\mathcal{B}^0$ has the quotient flow induced by $\Phi^t$. The conclusion now follows as the Brin group of the magnetic frame-extension flow on $P\hat{M}\to U\hat{M}$ is $\mathcal{B}^0$, by construction.
 \end{proof}

 Utilizing Lemma \ref{Brin-connected} and passing to a finite magnetomorphic covering of $M$ if needed, we may assume without loss of generality that $\mathcal{B}$ is connected without affecting the conclusion of Theorem~\ref{thm:vol-preserving}. In \cite[Proposition 3.3.6]{filip2024finiteness}, the authors utilize the equivalent of \eqref{Brin-dichotomy} in order to argue why the dimension of the fibers of $HM \rightarrow M$ have full dimension. If the coset space~$\mathcal{C}$ in \eqref{coset-space} is a point, the conclusion follows immediately, as the generic part of the set ${I = \pi_H^{-1}(Z)}$ corresponds to a nonempty closed subset of $\mathcal{C}$, and the set $I$ is saturated by the fibers of the projection to $HM$. Thus, in order to show that the fibers have full dimension, one only has to argue when $\mathcal{C} = [-1,1]$. The argument for this case is purely topological in \cite[Proposition~3.3.6]{filip2024finiteness}, and hence it carries over immediately to our setting:
\begin{equation}\label{full-dim}
 \parbox{.83\textwidth}{$\dim(Z(x)) = \dim (H_xM) = n-1$ for all $x\in M$, where we set $Z(x) = H_xM\cap Z$.}\end{equation}
We now conclude the proof of Theorem \ref{thm:vol-preserving2}, following the argument in \cite[Section 3.3.7]{filip2024finiteness}.

\begin{proof}[Proof of Theorem \ref{thm:vol-preserving2}]
Since $Z \subseteq HM$ is a closed real-analytic set of full dimension, it suffices to show that $Z$ is open. For every $z \in Z$, let $Z_z$ denote the germ of $Z$ at $z$. This has full dimension by \eqref{full-dim}, and therefore the complexified germ is an entire complex neighborhood of~$z$. By \cite[Chapter V, Proposition 1]{narasimhan2006introduction}, the germ $Z_z$ is an entire real neighborhood of $z$, and hence $Z$ is open.
\end{proof}

In the same way that Corollary \ref{cor:mag-cartan} followed from Theorem \ref{dyn-Cartan}, we have that Theorem \ref{thm:magnetic-finiteness} follows from Theorems \ref{thm:vol-preserving2}, and \ref{dyn-Cartan} applied to the rescaled magnetic system $\bigl(s^{-2}\g,s^{-2}\sigma\bigr)$ instead of the original system $(\g,\sigma)$, cf.\ Lemmas \ref{lem:reduction} and \cite{Paternains-94-convex}. In fact, as this argument shows, it suffices to assume that the \mbox{$s$-mag\-ne\-tic} flow is Anosov in Theorem \ref{thm:magnetic-finiteness} rather than the stronger assumption that the $s$-magnetic curvature is negative.

\begin{rem}\label{bugs}
 It is not immediately clear how to generalize \eqref{mag-covariant} to any semi-spray flow $\varphi^t$ on $UM$. Namely, the expression ${\rm D}W/{\rm d}t - X_V(\gamma, W)$ (cf.\ \eqref{hor-ver-comps}) does not define a linear connection on $\gamma^*(TM)$ because (i) $X_V(x,v)$ does not depend linearly on $v$, and (ii) the vector field $W$ is not restricted to have unit length. One can work around (i) by defining $\mathcal{D}^\varphi$ via the above expression, with $X_V(\gamma,W)$ replaced by the fiber-derivative of $X_V$ at $(\gamma,\dot{\gamma})$, evaluated at $W$, but a different issue arises: $\mathcal{D}^\varphi$ is not guaranteed to be a metric-compatible linear connection on $\gamma^*(TM)$, as the fiber-derivative of $X_V$ need not be skew-adjoint. This would be necessary to ensure that the $\mathcal{D}^\varphi$-parallel transport preserves the tangent spaces to totally $\varphi$-invariant submanifolds. For magnetic flows, however, $\mathcal{D}^\varphi$ agrees with $\mathcal{D}$ in \eqref{mag-covariant}.
\end{rem}

\subsection*{Acknowledgements}

The authors would like to thank David Fisher for suggesting that we explore the magnetic version of the result in \cite{filip2024finiteness} as well as for useful discussions. The authors would also like to thank Simion Filip and Ben Lowe for explaining various parts of their proof as well as many helpful discussions, Jeffrey Meyer for providing some very helpful intuition for arithmetic manifolds, and the referees for the helpful comments on an earlier version of this text. The first author was supported by the National Science Foundation under award No.\ DMS-2503020.

\pdfbookmark[1]{References}{ref}
\LastPageEnding

\end{document}